\documentclass[12pt]{article}
\usepackage{amsmath}
\usepackage{amsfonts}
\usepackage{amssymb}
\usepackage{amsthm}
\usepackage{stackrel}

\usepackage{color}

\usepackage[cp1251]{inputenc}

\usepackage{figsize}

\usepackage{graphics}

\newtheorem{ex}{Example}[section]

\newtheorem{df}{Definition}[section]
\newtheorem{thm}{Theorem}[section]

\newcommand{\R}{{\rm I}\kern-0.18em{\rm R}}
\newcommand{\1}{{\rm 1}\kern-0.25em{\rm I}}
\newcommand{\E}{{\rm I}\kern-0.18em{\rm E}}
\newcommand{\p}{{\rm I}\kern-0.18em{\rm P}}

\makeatletter

\def\@fnsymbol#1{\ensuremath{\ifcase#1\or a\or b\or c\or d\or \e\or f\or *\dagger 	\or \ddagger\ddagger \else\@ctrerr\fi}}

\title{Outliers and The Ostensibly Heavy Tails}
\author{Lev B. Klebanov\footnote{Department of Probability and Mathematical Statistics, Charles University, Prague, Czech Republic, e-mail: lev.klebanov@mff.cuni.cz} and Irina Volchenkova\footnote{Czech Technical University in Prague, Prague, Czech Republic}}
\date{}

\begin{document}
\maketitle

\begin{abstract}
The aim of the paper is to show that the presence of one possible type of outliers is not connected to that of heavy tails of the distribution. In contrary, typical situation for outliers appearance is the case of compact supported distributions.

\vspace{0.3cm}
\end{abstract}


\section{Introduction and preliminary considerations}\label{s1}
\setcounter{equation}{0}
In this paper we revise the concept of the {\textit {\textbf {outliers}}}, or, more precisely, outliers of the first type (in terminology of \cite{KAKK}). We found the contemporary notion rather vague, 
which motivates us to carefully dispute its meaning and connection with distribution tail behavior. Let us start by closely looking at the outliers definition. Usually, it is similar to that given in popular INTERNET encyclopedia Wikipedia: ``In statistics, an outlier is an observation point that is distant from other observations. An outlier may be due to variability in the measurement or it may indicate experimental error; the latter are sometimes excluded from the data." Obviously, the definition is given neither in mathematically nor statistically correct way. 
In particular, we found the description of "the point being distant" from other observations rather confusing. A little bit better seems to be a definition given on NISTA site:``An outlier is an observation that lies in abnormal distance from other values in a random sample from a population. In a sense, this definition leaves it up to the analyst (or a consensus process) to decide what will be considered abnormal. Before abnormal observations can be singled out, it is necessary to characterize normal observations." However, it has similar drawbacks. In our opinion, it is essential to specify some measurement unit of the considered distance and mainly the definition of the corresponding considered distance. 
Therefore we wish to conclude that the term outlier in such a setup is highly depended on the choice of topology and geometry of the space in which we consider our experiment. In the same manner, we found the term "experimental error" equally misleading. Say, outlier is an observation which is not connected to the particular experiment, and so this observation will not appear in the next experiment. However, the statistics is devoted to repeating the experiments, and such observations will be automatically excluded from further experiments and study. 

Now consider the  possibility that such "distant" observations remain appear\-ing in the repetitions of our experimental study. In that case,  we need to keep the observations attributed to the experiment. Therefore, it is misleading to label the observations as "errors". For example, the triggering event of occurrence of such observations can be caused by  the design of the particular experiment, i.e. the way how the experiment is designed does not capture the nature of  corresponding applied problem. As a result, some observations may appear as a natural phenomena seamlessly to the considered problem. However, there are no mathematical or statistical tools to recognize such a situation and so we are left with concluding that: {\textit{\textbf {such observations are in contradiction with mathematical model chosen to describe the practical model under study}}}. Of course, if some observations are in contradiction with one model, they may be in a good agreement with another model. And so we conclude that the notion of outliers is a model sensitive, i.e. the outlier needs to be associated with the concrete mathematical or statistical model. Based on our initial discussion, let us give the following definition.

\begin{df}\label{de1}
Consider a mathematical model of some real phenomena experiment. 
We say that an observation is the outlier for this particular model if it is "in contradiction" with the model, i.e. it is either impossible to have such an observation under the assumption that the model holds, or the probability to obtain such observation for the case of true model is extremely low. If the probability is very small yet non-zero, we denote the probability as $\beta$, we will call relevant observation the $\beta$-outlier. 
\end{df}
	
Definition \ref{de1} gives precise sense to the second part of the Wikipedia definition. However, it provides no connection to the first part. In the following sections of this paper we provide the arguments and explanations that some typical cases of the outliers appearance in the statistical modeling are closely connected with the  properly defined "the distant character" of them. These "proper definitions" provide meaningful suggestions to possible model modifications in order to include the outliers as an element of the new model. Note that some ideas of the modification of outliers definitions were already considered in \cite{Kl16}.
	
\section{Definition of distant outliers}\label{s2}
\setcounter{equation}{0}
	
In this section we explore the situation when some observations are "distant" from the others. What is the "unit of measurement" for such a distance? The natural way to start is to measure the distance of the observations to their mean value in terms of their standard deviation. This approach has been used in literature especially when study financial data. Such approach seems to be interesting. However, many authors made wrong consequences from it. Namely, they concluded that the presence of such outliers is connected to heavy-tailed character of underlying distribution of observations.

Let us give one example of such conclusions.
Arguments connected to outliers usually arise when considering such time series as the Dow Jones Industrial Average index (say, for the interesting Period from the 3 January 2000 to 31 December 2009), daily ISE-100 Index (November 2, 1987 - June 8, 2001) and many others (see, e.g., \cite{EK}, \cite{BMW}). The observed fact is that quite a lot of data not only fall outside the 99\% confidence interval on the mean, but also outside the range of $\pm 5\; \sigma$ from the average, or even $\pm10\;\sigma$. On the assumption of this circumstance the authors made two conclusions.

First (and absolutely correct) conclusion consists in the fact, that the observations under assumption of their independence and identical distribution are in contradiction with their normality.

The second conclusion is that the distribution of these random variables is heavy-tailed. This decision is not based on any mathematical justification. 

It is clear that for heavy-tailed distributions with infinite variance we cannot use the standard deviation as a measure of distance. However, in statistical considerations (especially, for unknown general distribution) one can change general characteristics by their empirical counterparts. More precisely, suppose that $X_1,X_2, \ldots ,X_n$ is a sequence of independent identically distributed (i.i.d.) random variables. Denote by
\[ \bar{x}_n=\frac{1}{n}\sum_{j=1}^{n}X_j, \;\; s^2_n=\frac{1}{n}\sum_{j=1}^{n}(X_j-\bar{x})^2\]
their empirical mean and empirical variance correspondingly. Let $k>0$ be a fixed number. 
Namely, let us estimate the following probability 
\begin{equation}\label{eq3}
p_n=\p \{|X-\bar{x}_n|/s_n>k\},
\end{equation}
\begin{df}\label{de2}
We say that the distribution of $X$ produces outliers of the first kind if the probability (\ref{eq3}) is high (say, higher than for normal distribution). 
\end{df}

Really, if one has a model based on Gaussian distribution then the presence of many observations with $p_n$ greater that for normal case contradicts to the model, and the observations appears to be outliers in the sense of our Definition \ref{de1}. Such approach was used in financial mathematics to show the Gaussian distribution provides bad model for corresponding data (see, for example, \cite{EK, BMW}).
	
The observations $X_j$ for which the inequality $|X_j-\bar{x}_n|/s_n>k$ holds appear to be outliers for Gaussian model. In some financial models the presence of them was considered as an argument for the existence of heavy tails for real distributions.  Unfortunately, this is not so (see \cite{KV, KlF, KTK}).
	
\begin{thm}\label{th1}(see \cite{KTK}) Suppose that $X_1,X_2, \ldots ,X_n$ is a sequence of i.i.d. r.v.s belonging to a domain of attraction of strictly stable random variable with index of stability $\alpha \in (0,2)$. Then
\begin{equation}\label{eq4}
\lim_{n \to \infty}p_n =0.
\end{equation}
\end{thm} 
	
\begin{proof}
Since $X_j,\; j=1, \ldots ,n$ belong to the domain of attraction of strictly stable random variable with index $\alpha <2$, it is also true that $X_1^2, \ldots , X_n^2$ belong to the domain of attraction of one-sided stable distribution with index $\alpha /2$. 
		
1) Consider at first the case $1<\alpha <2$. In this case, $\bar{x}_n \stackrel[n \to \infty]{}{\longrightarrow }a=\E X_1$ and
$s_n \stackrel[n \to \infty]{}{\longrightarrow }\infty$. We have
\[ \p\{|X_1 - \bar{x}_n| >k s_n\} = \p\{X_1> ks_n+\bar{x}_n\}+\p\{X_1< -ks_n+\bar{x}_n\} = \]
\[=\p\{X_1> ks_n+a+o(1)\}+\p\{X_1< -ks_n+a +o(1)\}  \stackrel[n \to \infty]{}{\longrightarrow } 0. \]
		
2) Suppose now that $0<\alpha < 1$. In this case, we have $\bar{x}_n \sim n^{1/\alpha -1}Y$ as $n \to \infty$. Here $Y$ is $\alpha$-stable random variable, and the sign $\sim$ is used for asymptotic equivalence. Similarly, 
\[ s^2_n =\frac{1}{n}\sum_{j=1}^{n}X_j^2 - \bar{x}^2_n \sim n^{2/\alpha -1}Z (1+o(1)), \]
where $Z$ has one-sided positive stable distribution with index $\alpha /2$. We have
\[ \p\{|X_1-\bar{x}_n|>k s_n\}=\p\{ (X_1-\bar{x}_n)^2>k s^2_n\}=\]
\[=\p\{X_1^2 > n^{2/\alpha -1}Z (1+o(1))\}  \stackrel[n \to \infty]{}{\longrightarrow } 0. \] 
		
3) In the case $\alpha =1$ we deal with Cauchy distribution. The proof for this case is very similar to that in the case 2). We omit the details.
\end{proof} 
	
From this Theorem it follows that (for sufficiently large $n$) many heavy-tailed distributions will not produce any outliers of the first kind. Moreover, now we see the the presence of outliers of the first kind is in contradiction with many models having heavy tailed distributions, particularly, with models involving stable distributions. By the way, word {\bf variability} is not defined precisely, too. It shows, that high variability may denote something different than high standard deviation. 

Theorem \ref{th1} shows that for many situations distributions with infinite variance do not produce outliers for sufficiently large values of the sample size $n$. It means, we may restrict ourselves with the case of distributions having finite second moment. Therefore, instead of (\ref{eq3}), it is better to consider corresponding characteristic of general distribution

\begin{equation}\label{eq5}
 p(\kappa;X)=\p\{|X-\E X |/\sigma_X > \kappa\},
\end{equation}
where $\E X$ is expectation of the random variable $X$ and $\sigma_X$ is its standard deviation. {\textit {\textbf {We shall say that the expression $p(\kappa;X)$ gives the probability to have outliers on the level $\kappa$}}}.

It is clear that if random variable $X$ has finite second moment then $p_n$ from  (\ref{eq3}) converges to $p(\kappa;X)$ defined by (\ref{eq5}). To see this it is sufficient to apply the law of large numbers to the sequences $X_j$ and $(X_j-\E X_j)^2$, $j = 1, 2, \ldots$.

Let us show that the presence of outliers is more likely for the distributions having compact support than for the case of non-compact one.
\begin{thm}\label{th2}
Let $X$ be a symmetric with respect to zero random variable. Suppose that $X$ has finite second moment and let $\kappa>1$ be a fixed number. Then there exists a random variable $X^*$ having compact support and such that
\[ p(\kappa;X)\leq p(\kappa;X^*). \]
\end{thm}
\begin{proof} If $X$ has a compact support then there is nothing to proof. Therefore, we have to consider the case when $X$ has non-compact support.
Choose positive number $A$ such that $A> \kappa \sigma_X$ and define
\[X^* = \begin{cases}
X \quad \text{for the case} \quad |X| <A, \\
0 \quad \text{for the case} \quad |X| \geq A.
\end{cases}\]
It is clear that
\[\sigma_X^2= \E X^2 = 2\Bigl( \int_{0}^{A} x^2 dF_X(x) + \int_{A}^{\infty} x^2 dF_X(x)\Bigr)      \]
and 
\[(\sigma_X^*)^2= \E (X^*)^2 = 2 \int_{0}^{A} x^2 dF_X(x) < \sigma_X^2.\]
Denote
\[q = \p\{|X| \geq A\}.\]
In view of non-compactness of $X$'s support we have $q>0$. Obviously, $\p\{X^* =0\} \geq q$.
Basing on all above we obtain
\[\p\{X \geq \sigma_X \kappa\} =\p\{\kappa \sigma_X \geq X < A\} + \p\{X \geq A\} = q/2 + \p\{\kappa \sigma_X \geq X^* < A\}\} -q = \]
\[= \p\{X^* \geq \kappa \sigma_{X}\}-q/2 \geq \p\{X^* \geq \kappa \sigma_{X^*}\}.\]
It just means that the probability to have outliers is greater for $X^*$ than for $X$:
\[ \p\{X^* \geq \kappa \sigma_{X^*}\} \geq \p\{X \geq \kappa \sigma_{X}\}.\]
\end{proof}
Now we see, that the distributions with compact support may have higher probability of outliers on the same level than the distributions with non-compact support. {\textit {\textbf {Therefore, the idea on connection of the presence of outliers with the heaviness of distributional tails is wrong}}}.

Let us now consider symmetric distribution with a compact support. The most interesting problem now is to describe the distributions having maximal possible probability of outliers. It appears that the problem had been considered in the literature (see, for example, \cite{KS}). Let us formulate  corresponding results below in the form suitable for us.

Let $X$ be a random variable whose distribution function has a compact support. Because the probability 
\[p(\kappa;X)=\p\{|X-\E X |/\sigma_X \geq \kappa\} \]
is location and scale invariant we can suppose that $\E X =0$ and $\p\{|X|\leq 1\}=1$. 

\begin{ex}\label{exm1}
Consider random variable $X$ taking three values: $-1,0,1$ with probabilities $p/2,1-p,p/2$, correspondingly. It is easy to see that $\E X =0$ and $\sigma_{X}=\sqrt{p}$. Therefore, $p(\kappa;X)=p$ for any $\kappa \in (1,1/\sqrt{p}]$.
\end{ex}
Particularly, for $\kappa =1/\sqrt{p}$ we obtain the greater value of outlier which may appears with probability $p$. Opposite, if $\kappa>1$ is fixed then we can define $p=1/\kappa^2$ and the random variable from Example \ref{exm1} will posses outliers of level $\kappa$ with probability $p$. From Selberg inequality (see, for example, \cite{KS}) it follows that random variable with outliers of the level $\geq \kappa$ appearing with the probability $\geq 1/\kappa^2$ coincides with that given in Example \ref{exm1} up to location and scale parameters. In other words, {\it the random variable with corresponding outlier properties is essentially unique}.
 
However, it seems that random variables similar to that from Example \ref{exm1} do appear in applied very rarely. Are there more practical examples with ``large enough" number of outliers? Corresponding example is based on Gauss inequality (see \cite{KS}).

\begin{ex}\label{exm2}
Consider random variable $X$ which takes zero value with probability $1-p$ ($p\in (0,1)$), and coinsides with uniformly distributed over $[-1,1]$ random variable $U$ with probability $p$. It is clear that $\E X =0$, $\sigma_{X}=\sqrt{p/3}$. Therefore, 
\[ \p\{|X|\geq \kappa \sigma_{X}\} =p \bigl(  1-\kappa \sqrt{p/3} \bigr). \]
Let us maximize this expression with respect to $p\in [0,1]$ assuming that $\kappa > 2/\sqrt{3}$. We obtain
\[ \max_{p \in [0,1]} (\p\{|X|\geq \kappa \sigma_{X}\} = \frac{4}{9 \kappa^2} \]
which attains for $p=4/(3 \kappa^2)$.
\end{ex}
From Gauss inequality (see \cite{KS}) it follows that this is upper bound in the class of unimodal distributions with finite variance. The extremal distribution constructed in Example \ref{exm2} is unique up to location and scale parameters.

Of course, the boundary for probability to have outliers on the level $\kappa$ is smaller in $9/4$ times for Example \ref{exm2} than that for Example \ref{exm1}. However, this probability is not too small comparing, say with Gaussian distribution. 

The distribution from Example \ref{exm2} looks also not too ``practical" in view of presence of an ``essential" mass at zero. However, it is clear that we may replace this mass by a pick of a density near origin. Of course, the probability of the presence of corresponding outliers will become smaller a little bit, but not too essential. Let us give some simulation results for such case. 

We simulated a sample of the volume $n=500$ from a mixture of two uniform distributions. First component of the mixture was following uniform law on interval $[-0.1,0.1]$. This component had weight $71/75$. The second component was following uniform distribution on interval $[-1,1]$. It had weight $4/75$. The sample points are shown on the Figure \ref{fig1} as blue points. Vertical lines are situated at positions $-5 \sigma$ and $5 \sigma$. We see some elements of the sample (at least 6 of them, that is 1.2\%) are outside the interval $[-5\sigma, 5\sigma]$.

\begin{figure}[h]
\centering
\hfil
\includegraphics[scale=0.8]{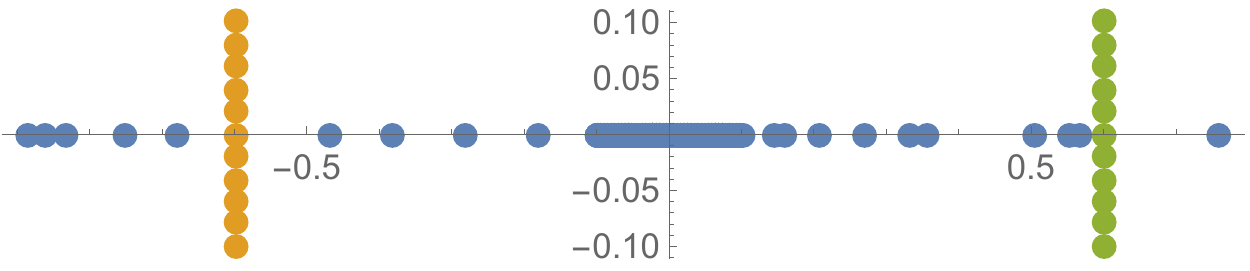}
\caption{{\sl A sample of 500 points from a mixture of two uniform distributions. Vertical lines are situated at positions $-5 \sigma$ and $5 \sigma$.}}\label{fig1}
\end{figure}

It is clear, that we observe some outliers for smooth enough distribution having a compact support. To understand what is ``typical form" of a distribution with not small probability of outliers of the level $\kappa$ let us consider another simulation.

We simulate $n=5000$ observations from mixture of two normal distributions. First component of the mixture is following standard normal distribution with the weight $p=4/75$. The second component is following normal distribution with zero mean and $\sigma=0.1$. The results are given on the Figure \ref{fig2}. We see that here is present rather large set of outliers on the level $\kappa =5$. Although the distribution has non-compact support its tail is not heavy. 

\begin{figure}[h]
\centering
\hfil
\includegraphics[scale=0.8]{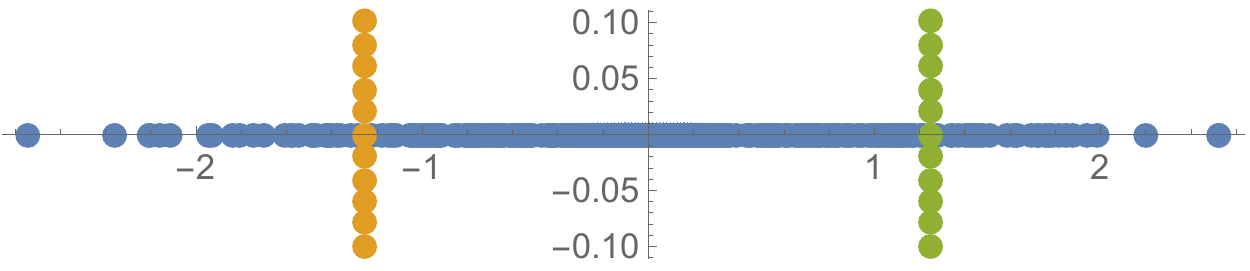}
\caption{{\sl A sample of 5000 points from a mixture of two normal distributions. Vertical lines are situated at positions $-5 \sigma$ and $5 \sigma$.}}\label{fig2}
\end{figure}

Let us look at corresponding histogram in different scale. Namely, we multiplied all data values by $20$. The resulting histogram is given on the Figure \ref{fig3}.

 \begin{figure}[h]
 \centering
 \hfil
 \includegraphics[scale=0.8]{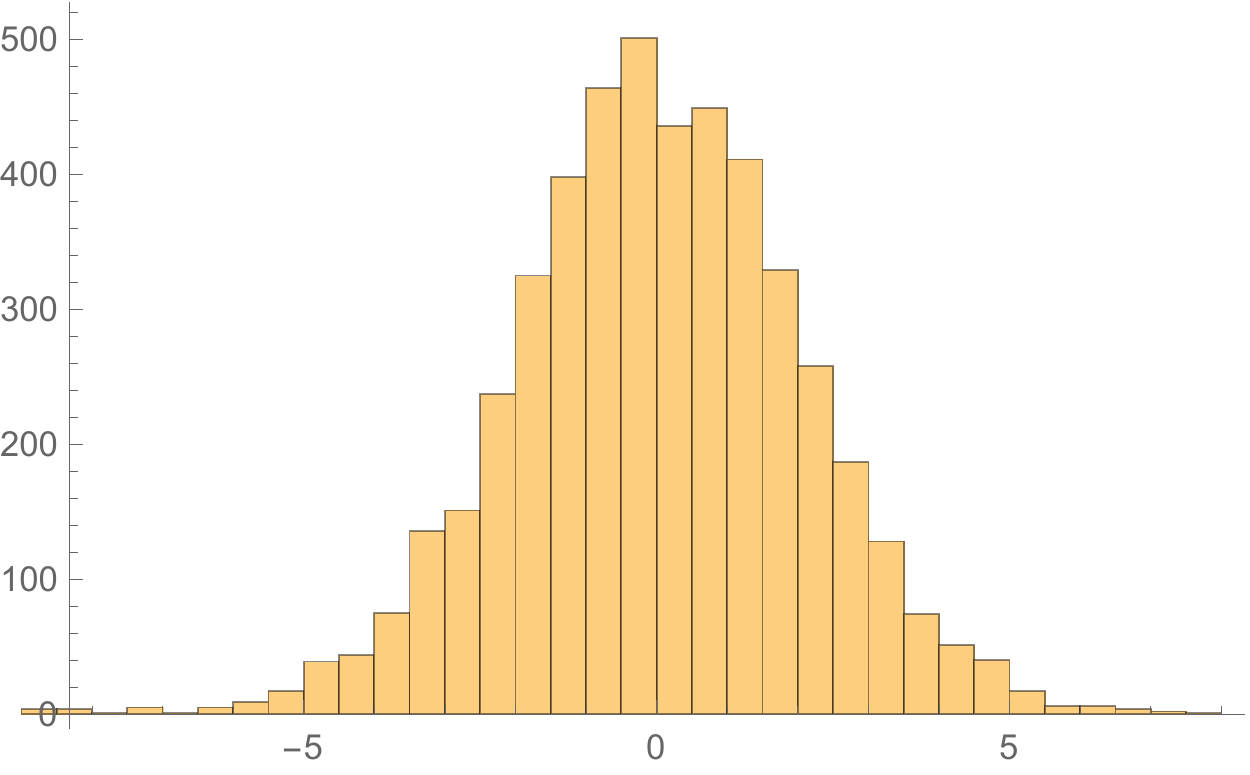}
 \caption{{\sl Histogram of sample of 5000 points (multiplied by 20) from a mixture of two normal distributions.}}\label{fig3}
 \end{figure}

The histogram looks as if the tails of the distribution were heavy. However, they are thin.

Consider one more situation. Let $X$ be a random variable having standard Gaussian distribution, and $A$ be a random variable with the density function 
\[ q(a;\alpha) = \frac{\alpha}{(a+1)^{\alpha+1}},\]
where $a\in [0,\infty)$ and $\alpha >2$ be a parameter. Consider the product $Y=X/A$. Its cumulative distribution function has form 
\[  G(x; \alpha) = \int_{0}^{\infty}\Phi (ax)q(a;\alpha) da. \]
It is clear that with growing $\alpha$ the tails of random variable $A$ become more thin. This implies that the tails of $Y$ become more heavy.
Computer calculations show that the probability of the event $\{ Y\geq \kappa \sigma(\alpha) \}$ for $\kappa =5$ has the form (as a function of $\alpha$)

 \begin{figure}[h]
 \centering
 \hfil
 \includegraphics[scale=0.8]{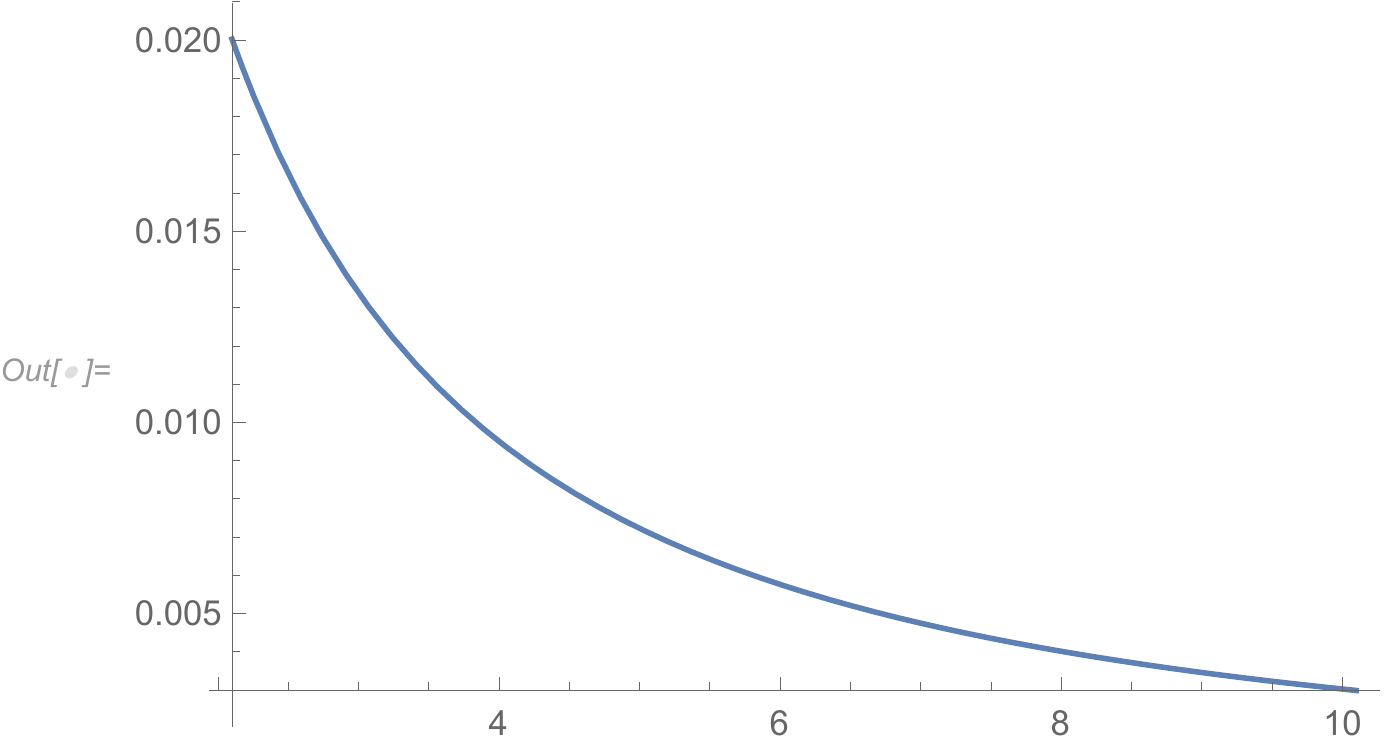}
 \caption{{\sl Probability of the event $\{ Y\geq 5 \sigma\}$ as a function of $\alpha \in (2,10]$ for the scale mixture of Gaussian distribution by means of random variable $A$ having Pareto distribution.}}\label{fig4}
 \end{figure}
From the Figure \ref{fig4} we see that the probability of outliers on the level $\kappa =5$  decreases with increasing $\alpha$, that is while the tails become more heavy.

Now we can resume that the opinion on the direct dependence between probability of outliers and heaviness of the tails is wrong and, probably, has to be changed by opposite, that is on inverse dependence. 

\section{Ostensible heavy tails}\label{s3}
\setcounter{equation}{0}

Let us try to understand what structure of distribution leads to high probability of outliers and why the general opinion consists in direct dependence between this probability and heaviness of the tails. Unfortunately, we cannot provide mathematically precise results in this connection, but will give intuitive clear explanation basing on the results obtained above.

The examples given above lead to the following. The distribution with high probability of outliers has the properties:
\begin{enumerate}
{\it \item[1.] High pike near mean value.
\item[2.] Thin or truncated ``far" tail\footnote{that is the part of distribution adjoing to infinity}.
\item[3.] The part of distributions body outside of the pike changes not too essentially with the remotion from the pike on a distance, but turns into thin tail after that. }
\end{enumerate}

In this situation, the main part of observations on such population will concentrate around pike. In view of this the empirical variance will not be large. According to the part outside of pike some observations will belong to this region, that is ``rather far" from the pike and, therefore, ``far" from mean value. This part of observations represents outliers. Obviously, thin or truncated tail plays no role, and we have rather high relative number of outliers.

In contrary, if the distribution has heavy tails (so that the variance is infinite) the observations belonging to ``far" tail make the variance large. It follows from the proof of Theorem \ref{th1} this grows is such quick that the probability of outliers on any positive level tends to zero as the number of observations tends to infinity. Such behavior is impossible for any finite distribution if the level of outliers is less than the distance from the border of support divided by the standard deviation of the distribution. So, intuitively, the presence of heavy tails is in contradiction with the presence of large level outliers.

In the situation when the general distribution possesses the properties $1. - 3.$ we say that is has {\textit {\textbf {ostensible heavy tails}}}.

But why the general opinion consists in direct dependence between this probability and heaviness of the tails? There are at least three reasons for that.
\begin{enumerate}
\item[\it A.] As it was mentioned in property $3.$ above, the part of distribution body lying outside of pike is not too thin while far enough from the pike. This seemingly is an effect of the heavy tail because some observations seems to be far from the main part of other observations. This is especially so if statistician does not compare the distance from other observation with the standard deviation.
\item[\it B.] Densities of symmetric stable distributions have rather high pike at the center of symmetry. Therefore, many observations are concentrated near the pike. If the sample size is not too large, the effect of heaviness of the tails is not too essential as if the statistician had sample from truncated distribution. So, it seems that there are some outliers. With growing sample size there will appear more and more observations from the tails making empirical standard higher and higher. This process leads to ``loss of outliers" and, in limit, results in zero probability of outliers presence. Now, we may conclude, that not large enough sample size may result in seemingly\footnote{``Seemingly" because we cannot pass to limit and are far from it in view of insufficiently sample size} outliers.
\item[\it C.] Although, some statisticians do not give precise definition of outliers they have in mind something different from that we discussed above. However, our point is that empirical understanding is not sufficient for solution any essential problem. Corresponding notions have to be formulated in precise terms. Some attempts to give different definitions of outliers were made in \cite{KAKK} and in \cite{KKK}.
\end{enumerate}

\section{Outliers and robust estimation}\label{s4} 
\setcounter{equation}{0} 

In introductions to many books on robust statistics there is written that one of the aims of robustness is to defend estimator from large outliers (see, for example, \cite{HRRS}). However, it appears that there are almost no conclusions on such defendence really given. Nobody shows the robust procedures allow to eliminate the influence of outliers. Just in opposite, in many publications it has been shown robust procedures defend the estimator from the presence of heavy tails. As it was indicated above, there is no direct connections between ouliers and heaviness of tails.

Let us consider the use of robust estimators from 
formally mathematical point of view. Let us discuss the problem on robust estimation of location parameter, or, more precisely, of symmetry center. Formally, the quality of robust estimators is guaranteed only asymptotically, that is for large sample size $n$. However, the presence of outliers asymptotically means that the tails are not heavy (see Theorem \ref{th1}) that the variance is finite. Therefore, the law of large numbers is valued and the sample mean is asymptotically consistent. This means we do not need any robust estimator. 

Of course, the statement above holds asymptotically only. And it is not clear how big should be $n$. However, it is not clear (in precise terms) for which $n$ robust estimator is good enough. Clearly, Example \ref{exm1} shows that the outliers may essentially change the quality of sample mean for the case of not large sample size. This means, we need to have some estimators with nice non-asymptotical properties. Our point of view is that there is needed a theory of robust non-asymptotically oriented estimators. Some initial aspects of such theory are given in \cite{KRF}. However, some more study is needed. We are planning to consider corresponding result in another paper. Let us note that classical approach to robust estimators is intuitively oriented on other definition of outliers. Probably, the definition given in \cite{KAKK} and \cite{KKK} more corresponds to classical theory of robust estimation. 

\section*{Acknowledgment}
The work was partially supported by Grant GA\v{C}R 16-03708S.

\end{document}